%% file: UniversesForCT.tex
\RequirePackage{snapshot}
\documentclass[12pt,a4paper]{article}

\usepackage[verbose]{geometry}
\input{PreludeA}

\usepackage{titling}
\usepackage{fancyhdr}

\hangsecnum

\title{Universes for category theory}
\author{Zhen~Lin Low}
\date{28 November 2014}

\begin{document}

\maketitle
\footpar{Department of Pure Mathematics and Mathematical Statistics, University of Cambridge, Cambridge, UK. \textsc{E-mail address}: \texttt{Z.L.Low@dpmms.cam.ac.uk}}
\input{Content}

\ifdraftdoc

\else
  \printbibliography
\fi

\end{document}

%% file: Content.tex
\begin{abstract}
The Grothendieck universe axiom asserts that every set is a member of some set-theoretic universe $\mbfU$ that is itself a set. One can then work with entities like the category of all $\mbfU$-sets or even the category of all locally $\mbfU$-small categories, where $\mbfU$ is an ``arbitrary but fixed'' universe, all without worrying about which set-theoretic operations one may legitimately apply to these entities. Unfortunately, as soon as one allows the possibility of changing $\mbfU$, one also has to face the fact that universal constructions such as limits or adjoints or Kan extensions could, in principle, depend on the parameter $\mbfU$. We will prove this is \emph{not} the case for adjoints of accessible functors between locally presentable categories (and hence, limits and Kan extensions), making explicit the idea that ``bounded'' constructions do not depend on the choice of $\mbfU$.
\end{abstract}

\input{Content-00}
\input{Content-01}
\input{Content-02}
\input{Content-03}
\input{Content-99}

%% file: Content-00.tex
\section*{Introduction}

In category theory it is often convenient to invoke a certain set-theoretic device commonly known as a `Grothendieck universe', but we shall say simply `universe', so as to simplify exposition and proofs by eliminating various circumlocutions involving cardinal bounds, proper classes \etc. In \citep[Exposé I, \Sect 0]{SGA4a}, the authors adopt the following \strong{universe axiom}:
\begin{quote}
For each set $x$, there exists a universe $\mbfU$ with $x \in \mbfU$.
\end{quote}
One then introduces a universe parameter $\mbfU$ and speaks of $\mbfU$-sets, locally $\mbfU$-small categories, and so on, with the proviso that $\mbfU$ is ``arbitrary''. We recall these notions in \Sect 1.

Having introduced universes into our ontology, it becomes necessary to ask whether an object with some universal property retains that property when we enlarge the universe. Though it sounds inconceivable, there do exist examples of badly-behaved constructions that are not stable under change-of-universe; for example, \citet{Waterhouse:1975} defined a functor $F : \cat{\CRing} \to \cat{\Setplus}$, where $\cat{\CRing}$ is the category of commutative rings in a universe $\mbfU$ and $\cat{\Setplus}$ is the category of $\mbfUplus$-sets for some universe $\mbfUplus$ with $\mbfU \in \mbfUplus$, such that the value of $F$ at any given commutative ring in $\mbfU$ does not depend on $\mbfU$, and yet the value of the fpqc sheaf associated with $F$ at the field $\mathbb{Q}$ depends on the size of $\mbfU$; more recently, \citet{Bowler:2012} has constructed an $\omega$-sequence of monads on $\cat{\Set}$ whose colimit depends on $\mbfU$, where $\cat{\Set}$ is the category of $\mbfU$-sets.

It is commonly said that `set-theoretic difficulties may be overcome with standard arguments using universes', but in light of the above remarks, it would appear \emph{prima facie} that the use of universes introduces new \emph{category-theoretic} difficulties! Of course, there are also standard arguments to overcome ostensible universe-dependence, and the purpose of the present paper is to analyse arguments that appeal to the ``boundedness'' of constructions. A classical technology for controlling size problems in category theory, due to \citet{Gabriel-Ulmer:1971}, Grothendieck and Verdier \citep[Exposé I, \Sect 9]{SGA4a}, and \citet{Makkai-Pare:1989}, is the notion of accessibility. We review this theory in \Sect 2 and apply it in \Sect 3 to study the stability of universal constructions under universe enlargement. Along the way we will identify semantic criteria for recognising inclusions of the form $\Ind[\kappa][\mbfU]{\mathbb{B}} \embedinto \Ind[\kappa][\mbfUplus]{\mathbb{B}}$, where $\mathbb{B}$ is $\mbfU$-small.

\subsection*{Acknowledgements}

The author acknowledges support from the Cambridge Commonwealth Trust and the Department of Pure Mathematics and Mathematical Statistics.

%% file: Content-01.tex
\section{Set theory}

\begin{dfn}
\label{dfn:universe}
A \strong{pre-universe} is a set $\mbfU$ satisfying these axioms:
\begin{enumerate}
\item If $x \in y$ and $y \in \mbfU$, then $x \in \mbfU$.

\item If $x \in \mbfU$ and $y \in \mbfU$ (but not necessarily distinct), then $\set{ x, y } \in \mbfU$.

\item If $x \in \mbfU$, then $\powerset{x} \in \mbfU$, where $\powerset{x}$ denotes the set of all subsets of $x$.

\item If $x \in \mbfU$ and $f : x \to \mbfU$ is a map, then $\bigcup_{i \in x} f \argp{i} \in \mbfU$.
\end{enumerate}
\needspace{2.5\baselineskip}
A \strong{universe} is a pre-universe $\mbfU$ with this additional property:
\begin{enumerate}[resume]
\item $\omega \in \mbfU$, where $\omega$ is the set of all finite (von Neumann) ordinals.
\end{enumerate}
\end{dfn}

\begin{example}
The empty set is a pre-universe, and with very mild assumptions, so is the set $\mathbf{HF}$ of all hereditarily finite sets.
\end{example}

\makenumpar 
For definiteness, we may take our base theory to be Mac Lane set theory, which is a weak subsystem of Zermelo--Fraenkel set theory with choice (ZFC); topos theorists may wish to note that Mac Lane set theory is comparable in strength to Lawvere's elementary theory of the category of sets (ETCS).\footnote{Mac Lane set theory is a subsystem of the theory $\mathrm{Z}_1$, and \citet{Mitchell:1972} has shown that one can construct a model of $\mathrm{Z}_1$ from any model of ETCS \emph{and vice versa}.} Readers interested in the details of Mac Lane set theory are referred to \citep{Mathias:2001}, but in practice, as long as one is working at all times \emph{inside some universe}, one may as well be working in ZFC. Indeed:

\begin{prop}
With the assumptions of Mac Lane set theory, any universe is a transitive model of ZFC.
\end{prop}
\begin{proof}
Let $\mbfU$ be a universe. By definition, $\mbfU$ is a transitive set containing pairs, power sets, unions, and $\omega$, so the axioms of extensionality, empty set, pairs, power sets, unions, choice, and infinity are all automatically satisfied. We must show that the axiom schemas of separation and replacement are also satisfied, and in fact it is enough to check that replacement is valid; but this is straightforward using axioms 2 and 4.
\end{proof}

\begin{prop}
\label{prop:universe.nesting}
In Mac Lane set theory:
\begin{enumerate}[(i)]
\item If $\mbfU$ is a non-empty pre-universe, then there exists a strongly inaccessible cardinal $\kappa$ such that the members of $\mbfU$ are all the sets of rank less than $\kappa$. Moreover, this $\kappa$ is the rank and the cardinality of $\mbfU$.

\item If $\mbfU$ is a universe and $\kappa$ is a strongly inaccessible cardinal such that $\kappa \in \mbfU$, then there exists a $\mbfU$-set $\mbfV_\kappa$ whose members are all the sets of rank less than $\kappa$, and $\mbfV_\kappa$ is a pre-universe.

\item If $\mbfU$ and $\mbfU'$ are pre-universes, then either $\mbfU \subseteq \mbfU'$ or $\mbfU' \subseteq \mbfU$; and if $\mbfU \subsetneqq \mbfU'$, then $\mbfU \in \mbfU'$.
\end{enumerate}
\end{prop}
\begin{proof} \exerproof
Omitted, but straightforward.
\end{proof}

\begin{cor}
In Mac Lane set theory plus the universe axiom, for each natural number $n$ in the meta-theory, the theory obtained by adding to ZFC the axiom that there are $n$ strongly inaccessible cardinals is consistent. \qed
\end{cor}

\begin{remark}
It is not clear what the consistency of ZFC plus the universe axiom is relative to Mac Lane set theory plus the universe axiom. For example, in ZFC, the universe axiom implies that there is a strictly ascending sequence of universes indexed by the class of all ordinals, but the construction of this sequence requires the axiom of replacement. 
\end{remark}

\makenumpar
We shall not require the full strength of the universe axiom in this paper, but we should at least assume that there are two universes $\mbfU$ and $\mbfUplus$, with $\mbfU \in \mbfUplus$. In what follows, the word `category' will always means a model of the first-order theory of categories \emph{inside} set theory, though not necessarily one that is a member of some universe.

\begin{dfn}
\label{dfn:U-set}
Let $\mbfU$ be a pre-universe. A \strong{$\mbfU$-set} is a member of $\mbfU$, a \strong{$\mbfU$-class} is a subset of $\mbfU$, and a \strong{proper $\mbfU$-class} is a $\mbfU$-class that is not a $\mbfU$-set.
\end{dfn}

\begin{lem}
A $\mbfU$-class $X$ is a $\mbfU$-set if and only if there exists a $\mbfU$-class $Y$ such that $X \in Y$. \qed
\end{lem}

\begin{prop}
If $\mbfU$ is a universe in Mac Lane set theory, then the collection of all $\mbfU$-classes is a transitive model of Morse--Kelley class--set theory (MK), and so is a transitive model of von Neumann--Bernays--Gödel class--set theory (NBG) in particular. 
\end{prop}
\begin{proof} \exerproof
Omitted.
\end{proof}

\begin{dfn}
\label{dfn:U-small.category}
A \strong{$\mbfU$-small category} is a category $\mathbb{C}$ such that $\ob \mathbb{C}$ and $\mor \mathbb{C}$ are $\mbfU$-sets. A \strong{locally $\mbfU$-small category} is a category $\mathcal{D}$ satisfying these conditions:
\begin{itemize}
\item $\ob \mathcal{D}$ and $\mor \mathcal{D}$ are $\mbfU$-classes, and

\item for all objects $x$ and $y$ in $\mathcal{D}$, the hom-set $\Hom[\mathcal{D}]{x}{y}$ is a $\mbfU$-set.
\end{itemize}
An \strong{essentially $\mbfU$-small category} is a category $\mathcal{D}$ for which there exist a $\mbfU$-small category $\mathbb{C}$ and a functor $\mathbb{C} \to \mathcal{D}$ that is fully faithful and essentially surjective on objects.
\end{dfn} 

\begin{dfn}
Let $\kappa$ be a regular cardinal. A \strong{$\kappa$-small category} is a category $\mathbb{C}$ such that $\mor \mathbb{C}$ has cardinality $< \kappa$. A \strong{finite category} is an $\aleph_0$-small category, \ie a category $\mathbb{C}$ such that $\mor \mathbb{C}$ is finite. A \strong{finite diagram} (\resp \strong{$\kappa$-small diagram}, \strong{$\mbfU$-small diagram}) in a category $\mathcal{C}$ is a functor $\mathbb{D} \to \mathcal{C}$ where $\mathbb{D}$ is a finite (\resp $\kappa$-small, $\mbfU$-small) category.
\end{dfn}

\begin{prop}
If $\mathbb{D}$ is a $\mbfU$-small category and $\mathcal{C}$ is a locally $\mbfU$-small category, then the functor category $\Func{\mathbb{D}}{\mathcal{C}}$ is locally $\mbfU$-small.
\end{prop}
\begin{proof}
Strictly speaking, this depends on the set-theoretic implementation of ordered pairs, categories, functors, \etc, but at the very least $\Func{\mathbb{D}}{\mathcal{C}}$ should be isomorphic to a locally $\mbfU$-small category.

In the context of $\Func{\mathbb{D}}{\mathcal{C}}$, we may regard functors $\mathbb{D} \to \mathcal{C}$ as being the pair consisting of the \emph{graph} of the object map $\ob \mathbb{D} \to \ob \mathcal{C}$ and the \emph{graph} of the morphism map $\mor \mathbb{D} \to \mor \mathcal{C}$, and these are $\mbfU$-sets by the $\mbfU$-replacement axiom. Similarly, if $F$ and $G$ are objects in $\Func{\mathbb{D}}{\mathcal{C}}$, then we may regard a natural transformation $\alpha : F \hoto G$ as being the triple $\tuple{F, G, A}$, where $A$ is the set of all pairs $\tuple{c, \alpha_c}$.
\end{proof}

One complication introduced by having multiple universes concerns the existence of (co)limits.

\begin{thm}[Freyd]
\label{thm:Freyd.complete.preorder}
Let $\mathcal{C}$ be a category and let $\kappa$ be a cardinal such that $\card{\mor \mathcal{C}} \le \kappa$. If $\mathcal{C}$ has products for families of size $\kappa$, then any two parallel morphisms in $\mathcal{C}$ must be equal.
\end{thm}
\begin{proof}
Suppose, for a contradiction, that $f, g : X \to Y$ are distinct morphisms in $\mathcal{C}$. Let $Z$ be the product of $\kappa$-many copies of $Y$ in $\mathcal{C}$. The universal property of products implies there are at least $2^\kappa$-many distinct morphisms $X \to Z$; but $\Hom[\mathcal{C}]{X}{Z} \subseteq \mor \mathcal{C}$, so this is an absurdity.
\end{proof}

\begin{dfn}
Let $\mbfU$ be a pre-universe. A \strong{$\mbfU$-complete} (\resp \strong{$\mbfU$-\allowhyphens cocom\-plete}) \strong{category} is a category $\mathcal{C}$ with the following property:
\begin{itemize}
\item For all $\mbfU$-small diagrams $A : \mathbb{D} \to \mathcal{C}$, a limit (\resp colimit) of $A$ exists in $\mathcal{C}$.
\end{itemize}
We may instead say $\mathcal{C}$ has all \strong{finite limits} (\resp \strong{finite colimits}) in the special case $\mbfU = \mathbf{HF}$.
\end{dfn}

\begin{prop}
\label{prop:limits.and.colimits}
Let $\mathcal{C}$ be a category and let $\mbfU$ be a non-empty pre-universe. The following are equivalent:
\begin{enumerate}[(i)]
\item $\mathcal{C}$ is $\mbfU$-complete.

\item $\mathcal{C}$ has all finite limits and products for all families of objects indexed by a $\mbfU$-set.

\item For each $\mbfU$-small category $\mathbb{D}$, there exists an adjunction
\[
\Delta \dashv {\prolim}_{\mathbb{D}} : \Func{\mathbb{D}}{\mathcal{C}} \to \mathcal{C}
\]
where $\Delta X$ is the constant functor with value $X$.
\end{enumerate}
\needspace{2.5\baselineskip}
Dually, the following are equivalent:
\begin{enumerate}[(i\prime)]
\item $\mathcal{C}$ is $\mbfU$-cocomplete.

\item $\mathcal{C}$ has all finite colimits and coproducts for all families of objects indexed by a $\mbfU$-set.

\item For each $\mbfU$-small category $\mathbb{D}$, there exists an adjunction
\[
{\indlim}_{\mathbb{D}} \dashv \Delta : \mathcal{C} \to \Func{\mathbb{D}}{\mathcal{C}}
\]
where $\Delta X$ is the constant functor with value $X$.
\end{enumerate}
\end{prop}
\begin{proof} \exerproof
This is a standard result; but we remark that we do require a sufficiently powerful form of the axiom of choice to pass from (ii) to (iii).
\end{proof}

\begin{thm}
\label{thm:foundations:stability.of.limits.and.colimits.of.sets}
Let $\mbfU$ be a pre-universe, let $\mbfUplus$ be a universe with $\mbfU \in \mbfUplus$, let $\cat{\Set}$ be the category of $\mbfU$-sets, and let $\cat{\Set^+}$ be the category of $\mbfUplus$-sets.
\begin{enumerate}[(i)]
\item If $X : \mathbb{D} \to \cat{\Set}$ is a $\mbfU$-small diagram, then there exist a limit and a colimit for $X$ in $\cat{\Set}$.

\item The inclusion $\cat{\Set} \embedinto \cat{\Set^+}$ is fully faithful and preserves limits and colimits for all $\mbfU$-small diagrams.
\end{enumerate}
\end{thm}
\begin{proof} \obviousproof
One can construct products, equalisers, coproducts, coequalisers, and hom-sets in a completely explicit way, making the preservation properties obvious. 
\end{proof}

\begin{cor}
The inclusion $\cat{\Set} \embedinto \cat{\Set^+}$ reflects limits and colimits for all $\mbfU$-small diagrams. \qed
\end{cor}

%% file: Content-02.tex
\section{Accessibility and ind-completions}

In this section, we recall some of the basic theory of locally presentable and accessible categories, making explicit any (apparent) dependence on the choice of universe.

\begin{thm}
\label{thm:filtered.colimits.in.set}
Let $\mbfU$ be a pre-universe, let $\cat{\Set}$ be the category of $\mbfU$-sets, and let $\kappa$ be any regular cardinal. Given a $\mbfU$-small category $\mathbb{D}$, the following are equivalent:
\begin{enumerate}[(i)]
\item $\mathbb{D}$ is a $\kappa$-filtered category.

\item The functor $\indlim_{\mathbb{D}} : \Func{\mathbb{D}}{\cat{\Set}} \to \cat{\Set}$ preserves limits for all diagrams that are simultaneously $\kappa$-small and $\mbfU$-small.
\end{enumerate}
\end{thm}
\begin{proof} \openproof
The claim (i) \implies (ii) is very well known, and the converse is an exercise in using the Yoneda lemma and manipulating limits and colimits for diagrams of representable functors; see Satz 5.2 in \citep{Gabriel-Ulmer:1971}.
\end{proof}

\begin{dfn}
Let $\kappa$ be a regular cardinal in a universe $\mbfUplus$, let $\mbfU$ be a pre-universe with $\mbfU \subseteq \mbfUplus$, and let $\cat{\Setplus}$ be the category of $\mbfUplus$-sets. A \strong{$\tuple{\kappa, \mbfU}$-compact object} in a locally $\mbfUplus$-small category $\mathcal{C}$ is an object $A$ such that the representable functor $\Hom[\mathcal{C}]{A}{\blank} : \mathcal{C} \to \cat{\Setplus}$ preserves colimits for all $\mbfU$-small $\kappa$-filtered diagrams. 
\end{dfn}

\begin{remark*}
It is more usual to say `$\lambda$-presentable object' instead of `$\lambda$-compact object', especially in algebraic contexts. This is (at least partially) justified by \autoref{prop:compact.objects.in.locally.presentable.categories}.
\end{remark*}

Though the above definition is stated using a pre-universe $\mbfU$ contained in a universe $\mbfUplus$, the following lemma shows there is no dependence on $\mbfUplus$.

\begin{lem}
\label{lem:elementary.condition.for.compactness}
Let $A$ be an object in a locally $\mbfUplus$-small category $\mathcal{C}$. The following are equivalent:
\begin{enumerate}[(i)]
\item $A$ is a $\tuple{\kappa, \mbfU}$-compact object in $\mathcal{C}$.

\item For all $\mbfU$-small $\kappa$-filtered diagrams $B : \mathbb{D} \to \mathcal{C}$, if $\lambda : B \hoto \Delta C$ is a colimiting cocone, then for any morphism $f : A \to C$, there exist an object $i$ in $\mathbb{D}$ and a morphism $f' : A \to B i$ in $\mathcal{C}$ such that $f = \lambda_i \circ f'$; and moreover if $f = \lambda_j \circ f''$ for some morphism $f'' : A \to B j$ in $\mathcal{C}$, then there exists an object $k$ and a pair of arrows $g : i \to k$, $h : i \to k$ in $\mathbb{D}$ such that $B g \circ f' = B h \circ f''$.
\end{enumerate}
\end{lem}
\begin{proof} \openproof
Use the explicit description of $\indlim_\mathbb{D} \Hom[\mathcal{C}]{A}{B}$ as a filtered colimit of sets; see Definition 1.1 in \citep{LPAC} or Proposition 5.1.3 in \citep{Borceux:1994b}.
\end{proof}

\begin{cor}
\label{cor:compact.objects.are.retracts}
Let $B : \mathbb{D} \to \mathcal{C}$ be a $\mbfU$-small $\kappa$-filtered diagram, and let $\lambda : B \hoto \Delta C$ be a colimiting cocone in $\mathcal{C}$. If $C$ is a $\tuple{\kappa, \mbfU}$-compact object in $\mathcal{C}$, then for some object $i$ in $\mathbb{D}$, $\lambda_i : B i \to C$ is a \emph{split} epimorphism. \qed
\end{cor}

\begin{lem}
\label{lem:compactness.rank}
Let $A$ be an object in a category $\mathcal{C}$.
\begin{enumerate}[(i)]
\item If $\mbfU$ is a pre-universe contained in a universe $\mbfUplus$ and $\kappa$ is a regular cardinal such that $A$ is $\tuple{\kappa, \mbfUplus}$-compact, then $A$ is $\tuple{\kappa, \mbfU}$-compact as well.

\item If $\kappa$ is a regular cardinal such that $A$ is $\tuple{\kappa, \mbfU}$-compact and $\lambda$ is any regular cardinal such that $\kappa \le \lambda$, then $A$ is also $\tuple{\lambda, \mbfU}$-compact.
\end{enumerate}
\end{lem}
\begin{proof} \obviousproof
Obvious.
\end{proof}

\begin{lem}
\label{lem:small.objects}
Let $\lambda$ be a regular cardinal in a universe $\mbfUplus$, and let $\mbfU$ be a pre-universe with $\mbfU \subseteq \mbfUplus$. If $B : \mathbb{D} \to \mathcal{C}$ is a $\lambda$-small diagram of $\tuple{\lambda, \mbfU}$-compact objects in a locally $\mbfUplus$-small category, then the colimit $\indlim_\mathbb{D} B$, if it exists, is a $\tuple{\lambda, \mbfU}$-compact object in $\mathcal{C}$.
\end{lem}
\begin{proof}
Use \autoref{thm:filtered.colimits.in.set} and the fact that $\Hom[\mathcal{C}]{\blank}{C} : \op{\mathcal{C}} \to \cat{\Setplus}$ maps colimits in $\mathcal{C}$ to limits in $\cat{\Setplus}$.
\end{proof}

\begin{cor}
\label{cor:retracts.of.compact.objects}
A retract of a $\tuple{\lambda, \mbfU}$-compact object is also a $\tuple{\lambda, \mbfU}$-compact object.
\end{cor}
\begin{proof}
Suppose $r : A \to B$ and $s : B \to A$ are morphisms in $\mathcal{C}$ such that $r \circ s = \id_B$. Then $e = s \circ r$ is an idempotent morphism and the diagram below
\[
\begin{tikzcd}
A
\rar[transform canvas={yshift=0.75ex}]{\id_A} 
\rar[transform canvas={yshift=-0.75ex}, swap]{e} &
A
\rar{r} &
B
\end{tikzcd}
\]
is a (split) coequaliser diagram in $\mathcal{C}$, so $B$ is $\tuple{\lambda, \mbfU}$-compact if $A$ is.
\end{proof}
%
%
%
%

\begin{dfn}
Let $\kappa$ be a regular cardinal in a universe $\mbfU$. A \strong{$\kappa$-accessible $\mbfU$-category} is a locally $\mbfU$-small category $\mathcal{C}$ satisfying the following conditions:
\begin{itemize}
\item $\mathcal{C}$ has colimits for all $\mbfU$-small $\kappa$-filtered diagrams.

\item There exists a $\mbfU$-set $\mathcal{G}$ such that every object in $\mathcal{G}$ is $\tuple{\kappa, \mbfU}$-compact and, for every object $B$ in $\mathcal{C}$, there exists a $\mbfU$-small $\kappa$-filtered diagram of objects in $\mathcal{G}$ with $B$ as its colimit in $\mathcal{C}$.
\end{itemize}
We write $\Kompakt[\kappa][\mbfU]{\mathcal{C}}$ for the full subcategory of $\mathcal{C}$ spanned by the $\tuple{\kappa, \mbfU}$-compact objects.
\end{dfn}

\begin{remark}
\Autoref{lem:small.objects} implies that, for each object $A$ in an accessible $\mbfU$-category, there exists a regular cardinal $\lambda$ in $\mbfU$ such that $A$ is $\tuple{\lambda, \mbfU}$-compact.
\end{remark}

\begin{thm}
\label{thm:ind-completion}
Let $\mathcal{C}$ be a locally $\mbfU$-small category, and let $\kappa$ be a regular cardinal in $\mbfU$. There exist a locally $\mbfU$-small category $\Ind[\kappa][\mbfU]{\mathcal{C}}$ and a functor $\gamma : \mathcal{C} \to \Ind[\kappa][\mbfU]{\mathcal{C}}$ with the following properties:
\begin{enumerate}[(i)]
\item The objects of $\Ind[\kappa][\mbfU]{\mathcal{C}}$ are $\mbfU$-small $\kappa$-filtered diagrams $B : \mathbb{D} \to \mathcal{C}$, and $\gamma$ sends an object $C$ in $\mathcal{C}$ to the diagram $\mathbf{1} \to \mathcal{C}$ with value $C$.

\item The functor $\gamma : \mathcal{C} \to \Ind[\kappa][\mbfU]{\mathcal{C}}$ is fully faithful, injective on objects, preserves all limits that exist in $\mathcal{C}$, and preserves all $\kappa$-small colimits that exist in $\mathcal{C}$.

\item $\Ind[\kappa][\mbfU]{\mathcal{C}}$ has colimits for all $\mbfU$-small $\kappa$-filtered diagrams.

\item For every object $C$ in $\mathcal{C}$, the object $\gamma C$ is $\tuple{\kappa, \mbfU}$-compact in $\Ind[\kappa][\mbfU]{\mathcal{C}}$, and for each $\mbfU$-small $\kappa$-filtered diagram $B : \mathbb{D} \to \mathcal{C}$, there is a canonical colimiting cocone $\gamma B \hoto \Delta B$ in $\Ind[\kappa][\mbfU]{\mathcal{C}}$.

\item If $\mathcal{D}$ is a category with colimits for all $\mbfU$-small $\kappa$-filtered diagrams, then for each functor $F : \mathcal{C} \to \mathcal{D}$, there exists a functor $\bar{F} : \Ind[\kappa][\mbfU]{\mathcal{C}} \to \mathcal{D}$ that preserves colimits for all $\mbfU$-small $\kappa$-filtered diagrams in $\Ind[\kappa][\mbfU]{\mathcal{C}}$ such that $\gamma \bar{F} = F$, and given any functor $\bar{G} : \Ind[\kappa][\mbfU]{\mathcal{C}} \to \mathcal{D}$ whatsoever, the induced map $\Nat{\bar{F}}{\bar{G}} \to \Nat{F}{\gamma \bar{G}}$ is a bijection.
\end{enumerate}
The category $\Ind[\kappa][\mbfU]{\mathcal{C}}$ is called the \strong{free $\tuple{\kappa, \mbfU}$-ind-completion} of $\mathcal{C}$, or the \strong{category of $\tuple{\kappa, \mbfU}$-ind-objects} in $\mathcal{C}$.
\end{thm}
\begin{proof} \openproof
See Corollary 6.4.14 in \citep{Borceux:1994a} and Theorem 2.26 in \citep{LPAC}; note that the fact that $\gamma$ preserves colimits for $\kappa$-small diagrams essentially follows from \autoref{thm:filtered.colimits.in.set}.
\end{proof}

\begin{prop}
\label{prop:compact.objects.in.accessible.categories}
\needspace{3\baselineskip}
Let $\mathbb{B}$ be a $\mbfU$-small category and let $\kappa$ be a regular cardinal in $\mbfU$.
\begin{enumerate}[(i)]
\item $\Ind[\kappa][\mbfU]{\mathbb{B}}$ is a $\kappa$-accessible $\mbfU$-category.

\item Every $\tuple{\kappa, \mbfU}$-compact object in $\Ind[\kappa][\mbfU]{\mathbb{B}}$ is a retract of an object of the form $\gamma B$, where $\gamma : \mathbb{B} \to \Ind[\kappa][\mbfU]{\mathbb{B}}$ is the canonical embedding.

\item $\Kompakt[\kappa][\mbfU]{\Ind[\kappa][\mbfU]{\mathbb{B}}}$ is an essentially $\mbfU$-small category.
\end{enumerate}
\end{prop}
\begin{proof}
(i). This claim more-or-less follows from the properties of $\Ind[\kappa][\mbfU]{\mathbb{B}}$ explained in the previous theorem.

\bigskip\noindent
(ii). Use \autoref{cor:retracts.of.compact.objects}.

\bigskip\noindent
(iii). Since $\mathbb{B}$ is $\mbfU$-small and $\Ind[\kappa][\mbfU]{\mathbb{B}}$ is locally $\mbfU$-small, claim (ii) implies that $\Kompakt[\kappa][\mbfU]{\Ind[\kappa][\mbfU]{\mathbb{B}}}$ must be essentially $\mbfU$-small.
\end{proof}

\begin{dfn}
\label{dfn:accessible.functor}
Let $\kappa$ be a regular cardinal in a universe $\mbfU$. A \strong{$\tuple{\kappa, \mbfU}$-\allowhyphens accessible functor} is a functor $F : \mathcal{C} \to \mathcal{D}$ such that
\begin{itemize}
\item $\mathcal{C}$ is a $\kappa$-accessible $\mbfU$-category, and

\item $F$ preserves all colimits for $\mbfU$-small $\kappa$-filtered diagrams.
\end{itemize}
An \strong{accessible functor} is a functor that is $\tuple{\kappa, \mbfU}$-accessible functor for some regular cardinal $\kappa$ in some universe $\mbfU$.
\end{dfn}

\begin{thm}[Classification of accessible categories]
\label{thm:classification.of.accessible.categories}
Let $\kappa$ be a regular cardinal in a universe $\mbfU$, and let $\mathcal{C}$ be a locally $\mbfU$-small category. The following are equivalent:
\begin{enumerate}[(i)]
\item $\mathcal{C}$ is a $\kappa$-accessible $\mbfU$-category.

\item The inclusion $\Kompakt[\kappa][\mbfU]{\mathcal{C}} \embedinto \mathcal{C}$ extends along the embedding $\gamma : \mathcal{C} \to \Ind[\kappa][\mbfU]{\mathcal{C}}$ to a $\tuple{\kappa, \mbfU}$-accessible functor $\Ind[\kappa][\mbfU]{\Kompakt[\kappa][\mbfU]{\mathcal{C}}} \to \mathcal{C}$ that is fully faithful and essentially surjective on objects.

\item There exist a $\mbfU$-small category $\mathbb{B}$ and a functor $\Ind[\kappa][\mbfU]{\mathbb{B}} \to \mathcal{C}$ that is fully faithful and essentially surjective on objects.
\end{enumerate}
\end{thm}
\begin{proof} \openproof
See Theorem 2.26 in \citep{LPAC}, or Theorem 5.3.5 in \citep{Borceux:1994b}.
\end{proof}

\begin{dfn}
\label{dfn:locally.presentable.category}
Let $\kappa$ be a regular cardinal in a universe $\mbfU$. A \strong{locally $\kappa$-presentable $\mbfU$-category} is a $\kappa$-accessible $\mbfU$-category that is also $\mbfU$-\allowhyphens cocom\-plete. A \strong{locally presentable $\mbfU$-category} is one that is a locally $\kappa$-presentable $\mbfU$-\allowhyphens category for some regular cardinal $\kappa$ in $\mbfU$, and we often say `locally finitely presentable' instead of `locally $\aleph_0$-presentable'.
\end{dfn}

\begin{lem}
\label{lem:accessibility.index.of.locally.presentable.categories}
Let $\mathcal{C}$ be a locally $\kappa$-presentable $\mbfU$-category.
\begin{enumerate}[(i)]
\item For any regular cardinal $\lambda$ in $\mbfU$, if $\kappa \le \lambda$, then $\mathcal{C}$ is a locally $\lambda$-present\-able $\mbfU$-category.

\item With $\lambda$ as above, if $F : \mathcal{C} \to \mathcal{D}$ is a $\tuple{\kappa, \mbfU}$-accessible functor, then it is also a $\tuple{\lambda, \mbfU}$-accessible functor.

\item If $\mbfUplus$ is any universe with $\mbfU \in \mbfUplus$, and $\mathcal{C}$ is a locally $\kappa$-presentable $\mbfUplus$-category, then $\mathcal{C}$ must be a preorder.
\end{enumerate}
\end{lem}
\begin{proof}
(i). See the remark after Theorem 1.20 in \citep{LPAC}, or Propositions 5.3.2 and 5.2.3 in \citep{Borceux:1994b}.

\bigskip\noindent
(ii). A $\lambda$-filtered diagram is certainly $\kappa$-filtered, so if $F$ preserves colimits for all $\mbfU$-small $\kappa$-filtered diagrams in $\mathcal{C}$, it must also preserve colimits for all $\mbfU$-small $\lambda$-filtered diagrams.

\bigskip\noindent
(iii). This is a corollary of \autoref{thm:Freyd.complete.preorder}.
\end{proof}

\begin{cor}
A category $\mathcal{C}$ is a locally presentable $\mbfU$-category for \emph{at most} one universe $\mbfU$, provided $\mathcal{C}$ is not a preorder.
\end{cor}
\begin{proof}
Use \autoref{prop:universe.nesting} together with the above lemma.
\end{proof}

\begin{thm}[Classification of locally presentable categories]
\label{thm:classification.of.locally.presentable.categories}
Let $\kappa$ be a regular cardinal in a universe $\mbfU$, let $\cat{\Set}$ be the category of $\mbfU$-sets, and let $\mathcal{C}$ be a locally $\mbfU$-small category. The following are equivalent:
\begin{enumerate}[(i)]
\item $\mathcal{C}$ is a locally $\kappa$-presentable $\mbfU$-category.

\item There exist a $\mbfU$-small category $\mathbb{B}$ that has colimits for $\kappa$-small diagrams and a functor $\Ind[\kappa][\mbfU]{\mathbb{B}} \to \mathcal{C}$ that is fully faithful and essentially surjective on objects.

\item The restricted Yoneda embedding $\mathcal{C} \to \Func{\op{\Kompakt[\kappa][\mbfU]{\mathcal{C}}}}{\cat{\Set}}$ is fully faithful, $\tuple{\kappa, \mbfU}$-accessible, and has a left adjoint.

\item There exist a $\mbfU$-small category $\mathbb{A}$ and a fully faithful $\tuple{\kappa, \mbfU}$-accessible functor $R : \mathcal{C} \to \Func{\mathbb{A}}{\cat{\Set}}$ such that $\mathbb{A}$ has limits for all $\kappa$-small diagrams, $R$ has a left adjoint, and $R$ is essentially surjective onto the full subcategory of functors $\mathbb{A} \to \cat{\Set}$ that preserve limits for all $\kappa$-small diagrams.

\item There exist a $\mbfU$-small category $\mathbb{A}$ and a fully faithful $\tuple{\kappa, \mbfU}$-accessible functor $R : \mathcal{C} \to \Func{\mathbb{A}}{\cat{\Set}}$ such that $R$ has a left adjoint.

\item $\mathcal{C}$ is a $\kappa$-accessible $\mbfU$-category and is $\mbfU$-complete.
\end{enumerate}
\end{thm}
\begin{proof} \openproof
See Proposition 1.27, Corollary 1.28, Theorem 1.46, and Corollary 2.47 in \citep{LPAC}, or Theorems 5.2.7 and 5.5.8 in \citep{Borceux:1994b}.
\end{proof}

\begin{remark}
If $\mathcal{C}$ is equivalent to $\Ind[\kappa][\mbfU]{\mathbb{B}}$ for some $\mbfU$-small category $\mathbb{B}$ that has colimits for all $\kappa$-small diagrams, then $\mathbb{B}$ must be equivalent to $\Kompakt[\kappa][\mbfU]{\mathcal{C}}$ by \autoref {prop:compact.objects.in.accessible.categories}. In other words, every locally $\kappa$-presentable $\mbfU$-category is, up to equivalence, the $\tuple{\kappa, \mbfU}$-ind-completion of an essentially unique $\mbfU$-small $\kappa$-cocomplete category.
\end{remark}

\begin{example}
Obviously, for any $\mbfU$-small category $\mathbb{A}$, the functor category $\Func{\mathbb{A}}{\cat{\Set}}$ is locally finitely presentable. More generally, one may show that for any $\kappa$-ary algebraic theory $\mathsf{T}$, possibly many-sorted, the category of $\mathsf{T}$-algebras in $\mbfU$ is a locally $\kappa$-presentable $\mbfU$-category. The above theorem can also be used to show that $\cat{\Cat}$, the category of $\mbfU$-small categories, is a locally finitely presentable $\mbfU$-small category.
\end{example}

\begin{cor}
\label{cor:filtered.colimits.preserve.tiny.limits}
Let $\mathcal{C}$ be a locally $\kappa$-presentable $\mbfU$-category. For any $\mbfU$-small $\kappa$-filtered diagram $\mathbb{D}$, ${\indlim}_{\mathbb{D}} : \Func{\mathbb{D}}{\mathcal{C}} \to \mathcal{C}$ preserves $\kappa$-small limits.
\end{cor}
\begin{proof}
The claim is certainly true when $\mathcal{C} = \Func{\mathbb{A}}{\cat{\Set}}$, by \autoref{thm:filtered.colimits.in.set}. In general, choose a $\tuple{\kappa, \mbfU}$-accessible fully faithful functor $R : \mathcal{C} \to \Func{\mathbb{A}}{\cat{\Set}}$ with a left adjoint, and simply note that $R$ creates limits for all $\mbfU$-small diagrams as well as colimits for all $\mbfU$-small $\kappa$-filtered diagrams.
\end{proof}

\begin{prop}
\label{prop:locally.presentable.functor.categories}
If $\mathcal{C}$ is a locally $\kappa$-presentable $\mbfU$-category and $\mathbb{D}$ is any $\mbfU$-small category, then the functor category $\Func{\mathbb{D}}{\mathcal{C}}$ is also a locally $\kappa$-presentable category.
\end{prop}
\begin{proof} \openproof
This can be proven using the classification theorem by noting that the 2-functor $\Func{\mathbb{D}}{\blank}$ preserves reflective subcategories, but see also Corollary 1.54 in \citep{LPAC}.
\end{proof}

\begin{prop}
\label{prop:compact.objects.in.locally.presentable.categories}
Let $\mathcal{C}$ be a locally $\kappa$-presentable $\mbfU$-category, and let $\lambda$ be a regular cardinal in $\mbfU$ with $\lambda \ge \kappa$. If $\mathcal{H}$ is a small full subcategory of $\mathcal{C}$ such that
\begin{itemize}
\item every $\tuple{\kappa, \mbfU}$-compact object in $\mathcal{C}$ is isomorphic to an object in $\mathcal{H}$, and

\item $\mathcal{H}$ is closed in $\mathcal{C}$ under colimits for $\lambda$-small diagrams,
\end{itemize}
then every $\tuple{\lambda, \mbfU}$-compact object in $\mathcal{C}$ is isomorphic to an object in $\mathcal{H}$. In particular, $\Kompakt[\lambda][\mbfU]{\mathcal{C}}$ is the smallest replete full subcategory of $\mathcal{C}$ containing $\Kompakt[\kappa][\mbfU]{\mathcal{C}}$ and closed in $\mathcal{C}$ under colimits for $\lambda$-small diagrams.
\end{prop}
\begin{proof}
Let $C$ be any $\tuple{\lambda, \mbfU}$-compact object in $\mathcal{C}$. Clearly, the comma category $\commacat{\mathcal{H}}{C}$ is a $\mbfU$-small $\lambda$-filtered category. Let $\mathcal{G} = \mathcal{H} \cap \Kompakt[\kappa][\mbfU]{\mathcal{C}}$. One can show that $\commacat{\mathcal{G}}{C}$ is a cofinal subcategory in $\commacat{\mathcal{H}}{C}$, and the classification theorem (\ref{thm:classification.of.locally.presentable.categories}) implies that tautological cocone on the canonical diagram $\commacat{\mathcal{G}}{C} \to \mathcal{C}$ is colimiting, so the tautological cocone on the diagram $\commacat{\mathcal{H}}{C} \to \mathcal{C}$ is also colimiting. Now, by \autoref{cor:compact.objects.are.retracts}, $C$ is a retract of an object in $\mathcal{H}$, and hence $C$ must be isomorphic to an object in $\mathcal{H}$, because $\mathcal{H}$ is closed under coequalisers.

For the final claim, note that $\Kompakt[\lambda][\mbfU]{\mathcal{C}}$ is certainly a replete full subcategory of $\mathcal{C}$ and contained in any replete full subcategory containing $\Kompakt[\kappa][\mbfU]{\mathcal{C}}$ and closed in $\mathcal{C}$ under colimits for $\lambda$-small diagrams, so we just have to show that $\Kompakt[\lambda][\mbfU]{\mathcal{C}}$ is also closed in $\mathcal{C}$ under colimits for $\lambda$-small diagrams; for this, we simply appeal to \autoref{lem:small.objects}.
\end{proof}

\begin{prop}
\label{prop:compact.objects.in.diagram.categories}
Let $\mathcal{C}$ be a locally $\kappa$-presentable $\mbfU$-category and let $\mathbb{D}$ be a $\mu$-small category in $\mbfU$. The $\tuple{\lambda, \mbfU}$-compact objects in $\Func{\mathbb{D}}{\mathcal{C}}$ are precisely the diagrams $\mathbb{D} \to \mathcal{C}$ that are componentwise $\tuple{\lambda, \mbfU}$-compact, so long as $\lambda \ge \max \set{ \kappa, \mu }$.
\end{prop}
\begin{proof}
First, note that Mac Lane's subdivision category\footnote{See \citep[\Chap IX, \Sect 5]{CWM}.} $\subdiv{\mathbb{D}}$ is also $\mu$-small, so $\Hom[\Func{\mathbb{D}}{\mathcal{C}}]{A}{B}$ is computed as the limit of a $\mu$-small diagram of hom-sets. More precisely, using end notation,
\[
\Hom[\Func{\mathbb{D}}{\mathcal{C}}]{A}{B} \cong \int_{d : \mathbb{D}} \Hom[\mathcal{C}]{A d}{B d}
\]
and so if $A$ is componentwise $\tuple{\lambda, \mbfU}$-compact, then $\Hom[\Func{\mathbb{D}}{\mathcal{C}}]{A}{\blank}$ preserves colimits for $\mbfU$-small $\lambda$-filtered diagrams, hence $A$ is itself $\tuple{\lambda, \mbfU}$-compact.

Now, suppose $A$ is a $\tuple{\lambda, \mbfU}$-compact object in $\Func{\mathbb{D}}{\mathcal{C}}$. Let $d$ be an object in $\mathbb{D}$, let $d^* : \Func{\mathbb{D}}{\mathcal{C}} \to \mathcal{C}$ be evaluation at $d$, and let $d_* : \mathcal{C} \to \Func{\mathbb{D}}{\mathcal{C}}$ be the right adjoint, which is explicitly given by
\[
\parens{d_* C} \argp{d'} = \Hom[\mathbb{D}]{d'}{d} \cotens C
\]
where $\cotens$ is defined by following adjunction:
\[
\Hom[\Set]{X}{\Hom[\mathcal{C}]{C'}{C}} \cong \Hom[\mathcal{C}]{C}{X \cotens C}
\]
The unit $\eta_A : A \to d_* d^* A$ is constructed using the universal property of $\cotens$ in the obvious way, and the counit $\epsilon_C : d^* d_* C \to C$ is the projection $\Hom[\mathbb{D}]{d}{d} \cotens C \to C$ corresponding to $\id_d \in \Hom[\mathbb{D}]{d}{d}$. Since $\mathcal{C}$ is a locally $\lambda$-presentable $\mbfU$-category, there exist a $\mbfU$-small $\lambda$-filtered diagram $B : \mathbb{J} \to \mathcal{C}$ consisting of $\tuple{\lambda, \mbfU}$-compact objects in $\mathcal{C}$ and a colimiting cocone $\alpha : B \hoto \Delta d^* A$, and since each $\Hom[\mathbb{D}]{d'}{d}$ has cardinality less than $\mu$, the cocone $d_* \alpha : d_* B \hoto \Delta d_* d^* A$ is also colimiting, by \autoref{cor:filtered.colimits.preserve.tiny.limits}. \Autoref{lem:elementary.condition.for.compactness} then implies $\eta_A : A \to d_* d^* A$ factors through $d_* \alpha_j : d_* \argp{B j} \to d_* d^* A$ for some $j$ in $\mathbb{J}$, say
\[
\eta_A = d_* \alpha_j \circ \sigma
\]
for some $\sigma : A \to d_* B j$. But then, by the triangle identity,
\[
\id_{A d}
= \epsilon_{A d} \circ d^* \eta_A 
= \epsilon_{A d} \circ d^* d_* \alpha_j \circ d^* \sigma 
= \alpha_j \circ \epsilon_{B j} \circ d^* \sigma
\]
and so $\alpha_j : B j \to A d$ is a split epimorphism, hence $A d$ is a $\tuple{\lambda, \mbfU}$-compact object, by \autoref{cor:retracts.of.compact.objects}.
\end{proof}

\begin{remark}
The claim in the above proposition can fail if $\mu > \lambda \ge \kappa$. For example, we could take $\mathcal{C} = \cat{\Set}$, with $\mathbb{D}$ being the set $\omega$ considered as a discrete category; then the terminal object in $\Func{\mathbb{D}}{\cat{\Set}}$ is componentwise finite, but is not itself an $\aleph_0$-compact object in $\Func{\mathbb{D}}{\cat{\Set}}$.
\end{remark}

\begin{lem}
Let $\kappa$ and $\lambda$ be regular cardinals in a universe $\mbfU$, with $\kappa \le \lambda$.
\begin{enumerate}[(i)]
\item If $\mathcal{D}$ is a locally $\lambda$-presentable $\mbfU$-category, $\mathcal{C}$ is a locally $\mbfU$-small category, and $G : \mathcal{D} \to \mathcal{C}$ is a $\tuple{\lambda, \mbfU}$-accessible functor that preserves limits for all $\mbfU$-small diagrams in $\mathcal{C}$, then, for any $\tuple{\kappa, \mbfU}$-compact object $C$ in $\mathcal{C}$, the comma category $\commacat{C}{G}$ has an initial object.

\item If $\mathcal{C}$ is a locally $\kappa$-presentable $\mbfU$-category, $\mathcal{D}$ is a locally $\mbfU$-small category, and $F : \mathcal{C} \to \mathcal{D}$ is a functor that preserves colimits for all $\mbfU$-small diagrams in $\mathcal{C}$, then, for any object $D$ in $\mathcal{D}$, the comma category $\commacat{F}{D}$ has a terminal object.
\end{enumerate}
\end{lem}
\begin{proof}
(i). Let $\mathcal{F}$ be the full subcategory of $\commacat{C}{G}$ spanned by those $\tuple{D, g}$ where $D$ is a $\tuple{\lambda, \mbfU}$-compact object in $\mathcal{D}$. $G$ preserves colimits for all $\mbfU$-small $\lambda$-filtered diagrams, so, by \autoref{lem:elementary.condition.for.compactness}, $\mathcal{F}$ must be a weakly initial family in $\commacat{C}{G}$. \Autoref{prop:compact.objects.in.accessible.categories} implies $\mathcal{F}$ is an essentially $\mbfU$-small category, and since $\mathcal{D}$ has limits for all $\mbfU$-small diagrams and $G$ preserves them, $\commacat{C}{G}$ is also $\mbfU$-complete. Thus,  the inclusion $\mathcal{F} \embedinto \commacat{C}{G}$ has a limit, and it can be shown that this is an initial object in $\commacat{C}{G}$.\footnote{See Theorem 1 in \citep[\Chap X, \Sect 2]{CWM}.}

\bigskip\noindent
(ii). Let $\mathcal{G}$ be the full subcategory of $\commacat{F}{D}$ spanned by those $\tuple{C, f}$ where $C$ is a $\tuple{\kappa, \mbfU}$-compact object in $\mathcal{C}$; note that \autoref{prop:compact.objects.in.accessible.categories} implies $\mathcal{G}$ is an essentially $\mbfU$-small category. Since $\mathcal{C}$ has colimits for all $\mbfU$-small diagrams and $F$ preserves them, $\commacat{F}{D}$ is also $\mbfU$-cocomplete.\footnote{See the Lemma in \citep[\Chap V, \Sect 6]{CWM}.} Let $\tuple{C, f}$ be a colimit for the inclusion $\mathcal{G} \embedinto \commacat{F}{D}$. It is not hard to check that $\tuple{C, f}$ is a weakly terminal object in $\commacat{F}{D}$, so the formal dual of Freyd's initial object lemma\footnote{See Theorem 1 in \citep[\Chap V, \Sect 6]{CWM}.} gives us a terminal object in $\commacat{F}{D}$; explicitly, it may be constructed as the joint coequaliser of all the endomorphisms of $\tuple{C, f}$.
\end{proof}

\begin{thm}[Accessible adjoint functor theorem]
\label{thm:accessible.adjoints}
Let $\kappa$ and $\lambda$ be regular cardinals in a universe $\mbfU$, with $\kappa \le \lambda$, let $\mathcal{C}$ be a locally $\kappa$-presentable $\mbfU$-category, and let $\mathcal{D}$ be a locally $\lambda$-presentable $\mbfU$-category.

Given a functor $F : \mathcal{C} \to \mathcal{D}$, the following are equivalent:
\begin{enumerate}[(i)]
\item $F$ has a right adjoint $G : \mathcal{D} \to \mathcal{C}$, and $G$ is a $\tuple{\lambda, \mbfU}$-accessible functor.

\item $F$ preserves colimits for all $\mbfU$-small diagrams and sends $\tuple{\kappa, \mbfU}$-compact objects in $\mathcal{C}$ to $\tuple{\lambda, \mbfU}$-compact objects in $\mathcal{D}$.

\item $F$ has a right adjoint and sends $\tuple{\kappa, \mbfU}$-compact objects in $\mathcal{C}$ to $\tuple{\lambda, \mbfU}$-\allowhyphens compact objects in $\mathcal{D}$. 
\end{enumerate}
On the other hand, given a functor $G : \mathcal{D} \to \mathcal{C}$, the following are equivalent:
\begin{enumerate}[resume*]
\item $G$ has a left adjoint $F : \mathcal{C} \to \mathcal{D}$, and $F$ sends $\tuple{\kappa, \mbfU}$-compact objects in $\mathcal{C}$ to $\tuple{\lambda, \mbfU}$-compact objects in $\mathcal{D}$.

\item $G$ is a $\tuple{\lambda, \mbfU}$-accessible functor and preserves limits for all $\mbfU$-small diagrams.

\item $G$ is a $\tuple{\lambda, \mbfU}$-accessible functor and there is a functor $F_0 : \Kompakt[\kappa][\mbfU]{\mathcal{C}} \to \mathcal{D}$ with hom-set bijections
\[
\Hom[\mathcal{C}]{C}{G D} \cong \Hom[\mathcal{D}]{F_0 C}{D}
\]
natural in $D$ for each $\tuple{\kappa, \mbfU}$-compact object $C$ in $\mathcal{C}$, where $D$ varies over the objects in $\mathcal{D}$.
\end{enumerate}
\end{thm}
\begin{proof}
We will need to refer back to the details of the proof of this theorem later, so here is a sketch of the constructions involved.

\bigskip\noindent
(i) \implies (ii). If $F$ is a left adjoint, then $F$ certainly preserves colimits for all $\mbfU$-small diagrams. Given a $\tuple{\kappa, \mbfU}$-compact object $C$ in $\mathcal{C}$ and a $\mbfU$-small $\lambda$-filtered diagram $B : \mathbb{J} \to \mathcal{D}$, observe that
\begin{multline*}
\Hom[\mathcal{D}]{F C}{\textstyle \indlim_\mathbb{J} B}
  \cong \Hom[\mathcal{C}]{C}{\textstyle G \indlim_\mathbb{J} B} 
  \cong \Hom[\mathcal{C}]{C}{\textstyle \indlim_\mathbb{J} G B} \\
  \cong \textstyle \indlim_\mathbb{J} \Hom[\mathcal{C}]{C}{G B} 
  \cong \textstyle \indlim_\mathbb{J} \Hom[\mathcal{C}]{F C}{B}
\end{multline*}
and thus $F C$ is indeed a $\tuple{\lambda, \mbfU}$-compact object in $\mathcal{D}$ .

\bigskip\noindent
(ii) \implies (iii). It is enough to show that, for each object $D$ in $\mathcal{D}$, the comma category $\commacat{F}{D}$ has a terminal object $\tuple{G D, \epsilon_D}$;\footnote{See Theorem 2 in \citep[\Chap IV, \Sect 1]{CWM}.} but this was done in the previous lemma.

\bigskip\noindent
(iii) \implies (i). Given a $\tuple{\kappa, \mbfU}$-compact object $C$ in $\mathcal{C}$ and a $\mbfU$-small $\lambda$-filtered diagram $B : \mathbb{J} \to \mathcal{D}$, observe that
\begin{multline*}
\Hom[\mathcal{C}]{C}{\textstyle \textstyle G \indlim_\mathbb{J} B} 
  \cong \Hom[\mathcal{D}]{F C}{\textstyle \indlim_\mathbb{J} B}
  \cong \textstyle \indlim_\mathbb{J} \Hom[\mathcal{C}]{F C}{B} \\ 
  \cong \textstyle \indlim_\mathbb{J} \Hom[\mathcal{C}]{C}{G B}
  \cong \Hom[\mathcal{C}]{C}{\textstyle \indlim_\mathbb{J} G B}
\end{multline*}
because $F C$ is a $\tuple{\lambda, \mbfU}$-compact object in $\mathcal{D}$; but \autoref{thm:classification.of.locally.presentable.categories} says the restricted Yoneda embedding $\mathcal{C} \to \Func{\op{\Kompakt[\kappa][\mbfU]{\mathcal{C}}}}{\cat{\Set}}$ is fully faithful, so this is enough to conclude that $G$ preserves colimits for $\mbfU$-small $\lambda$-filtered diagrams.

\bigskip\noindent
(iv) \implies (v). If $G$ is a right adjoint, then $G$ certainly preserves limits for all $\mbfU$-small diagrams; the rest of this implication is just (iii) \implies (i). 

\bigskip\noindent
(v) \implies (vi). It is enough to show that, for each $\tuple{\kappa, \mbfU}$-compact object $C$ in $\mathcal{C}$, the comma category $\commacat{C}{G}$ has an initial object $\tuple{F_0 C, \eta_C}$; but this was done in the previous lemma. It is clear how to make $F_0$ into a functor $\Kompakt[\kappa][\mbfU]{\mathcal{C}} \to \mathcal{D}$.

\bigskip\noindent
(vi) \implies (iv). We use theorems \ref{thm:ind-completion} and \ref{thm:classification.of.locally.presentable.categories} to extend $F_0 : \Kompakt[\kappa][\mbfU]{\mathcal{C}} \to \mathcal{D}$ along the inclusion $\Kompakt[\kappa][\mbfU]{\mathcal{C}} \embedinto \mathcal{C}$ to get $\tuple{\kappa, \mbfU}$-accessible functor $F : \mathcal{C} \to \mathcal{D}$. We then observe that, for any $\mbfU$-small $\kappa$-filtered diagram $A : \mathbb{I} \to \mathcal{C}$ of $\tuple{\kappa, \mbfU}$-compact objects in $\mathcal{C}$,
\begin{multline*}
\Hom[\mathcal{C}]{\textstyle \indlim_{\mathbb{I}} A}{G D}
  \cong \textstyle \prolim_{\mathbb{I}} \Hom[\mathcal{C}]{A}{G D}
  \cong \textstyle \prolim_{\mathbb{I}} \Hom[\mathcal{C}]{F_0 A}{D} \\
  \cong \Hom[\mathcal{C}]{\textstyle \indlim_{\mathbb{I}} F A}{D}
  \cong \Hom[\mathcal{C}]{\textstyle F \indlim_{\mathbb{I}} A}{D}
\end{multline*}
is a series of bijections natural in $D$, where $D$ varies in $\mathcal{D}$; but $\mathcal{C}$ is a locally $\kappa$-presentable $\mbfU$-category, so this is enough to show that $F$ is a left adjoint of $G$. The remainder of the claim is a corollary of (i) \implies (ii).
\end{proof}

%% file: Content-03.tex
\section{Change of universe}

Many of the universal properties of interest concern adjunctions, so that is where we begin.

\begin{dfn}
Let $F \dashv G : \mathcal{D} \to \mathcal{C}$ and $F' \dashv G' : \mathcal{D}' \to \mathcal{C}'$ be adjunctions, and let $H : \mathcal{C} \to \mathcal{C}'$ and $K : \mathcal{D} \to \mathcal{D}'$ be functors. The \strong{mate} of a natural transformation $\alpha : H G \hoto G' K$ is the natural transformation
\[
\epsilon' K F \bcirc F' \alpha F \bcirc F' H \eta : F' H \hoto K F
\]
where $\eta : \id_{\mathcal{C}} \hoto G F$ is the unit of $F \dashv G$ and $\epsilon' : F' G' \hoto \id_{\mathcal{D}}$ is the counit of $F' \dashv G'$; dually, the \strong{mate} of a natural transformation $\beta : F' H \hoto K F$ is the natural transformation
\[
G' K \epsilon \bcirc G' \beta G \bcirc \eta' H G : H G \hoto G' K
\]
where $\eta' : \id_{\mathcal{C}'} \hoto G' F'$ is the unit of $F' \dashv G'$ and $\epsilon : F G \hoto \id_{\mathcal{D}}$ is the counit of $F \dashv G$.
\end{dfn}

\begin{dfn}
Given a diagram of the form
\[
\begin{tikzpicture}[commutative diagrams/every diagram]
\matrix[matrix of math nodes, name=m]{
	\mathcal{D} &
	\mathcal{D}' \\
	\mathcal{C} &
	\mathcal{C}' \\
};
\path[commutative diagrams/.cd, every arrow, every label]
	(m-1-1) edge node[swap] {$G$} (m-2-1)
	(m-1-1) edge node {$K$} (m-1-2)
	(m-1-2) edge node {$G'$} (m-2-2)
	(m-2-1) edge node[swap] {$H$} (m-2-2);
\path (m-2-1) -- (m-1-2)
	node[near start](a-1){} node[near end](b-1){};
\path[commutative diagrams/.cd, every arrow, every label]
	(a-1) edge[commutative diagrams/Rightarrow] node{$\alpha$} (b-1);
\end{tikzpicture}
\]
where $\alpha : H G \hoto G' K$ is a natural isomorphism, $F \dashv G$ and $F' \dashv G'$, we say the diagram satisfies the \strong{left Beck--Chevalley condition} if the mate of $\alpha$ is also a natural isomorphism. Dually, given a diagram of the form
\[
\begin{tikzpicture}[commutative diagrams/every diagram]
\matrix[matrix of math nodes, name=m]{
	\mathcal{C} &
	\mathcal{C}' \\
	\mathcal{D} &
	\mathcal{D}' \\
};
\path[commutative diagrams/.cd, every arrow, every label]
	(m-1-1) edge node[swap] {$F$} (m-2-1)
	(m-1-1) edge node {$H$} (m-1-2)
	(m-1-2) edge node {$F'$} (m-2-2)
	(m-2-1) edge node[swap] {$K$} (m-2-2);
\path (m-1-2) -- (m-2-1)
	node[near start](a-1){} node[near end](b-1){};
\path[commutative diagrams/.cd, every arrow, every label]
	(a-1) edge[commutative diagrams/Rightarrow] node{$\beta$} (b-1);
\end{tikzpicture}
\]
where $\beta : F' H \hoto K F$ is a natural isomorphism, $F \dashv G$ and $F' \dashv G'$, we say the diagram satisfies the \strong{right Beck--Chevalley condition} if the mate of $\beta$ is also a natural isomorphism.
\end{dfn}

\begin{lem}
\label{lem:local.Beck-Chevalley.condition}
Given a diagram of the form
\[
\begin{tikzpicture}[commutative diagrams/every diagram]
\matrix[matrix of math nodes, name=m]{
	\mathcal{D} &
	\mathcal{D}' \\
	\mathcal{C} &
	\mathcal{C}' \\
};
\path[commutative diagrams/.cd, every arrow, every label]
	(m-1-1) edge node[swap] {$G$} (m-2-1)
	(m-1-1) edge node {$K$} (m-1-2)
	(m-1-2) edge node {$G'$} (m-2-2)
	(m-2-1) edge node[swap] {$H$} (m-2-2);
\path (m-2-1) -- (m-1-2)
	node[near start](a-1){} node[near end](b-1){};
\path[commutative diagrams/.cd, every arrow, every label]
	(a-1) edge[commutative diagrams/Rightarrow] node{$\alpha$} (b-1);
\end{tikzpicture}
\]
where $\alpha : H G \hoto G' K$ is a natural isomorphism, $F \dashv G$ and $F' \dashv G'$, the diagram satisfies the left Beck--Chevalley condition if and only if, for every object $C$ in $\mathcal{C}$, the functor $\commacat{C}{G} \to \commacat{H C}{G'}$ sending an object $\tuple{D, f}$ in the comma category $\commacat{C}{G}$ to the object $\tuple{K D, \alpha_D \circ H f}$ in $\commacat{H C}{G'}$ preserves initial objects.
\end{lem}
\begin{proof}
We know $\tuple{F C, \eta_C}$ is an initial object of $\commacat{C}{G}$ and $\tuple{F' H C, \eta'_{H C}}$ is an initial object of $\commacat{H C}{G'}$, so there is a unique morphism $\beta_C : F' H C \to K F C$ such that $G' \beta_C \circ \eta'_{H C} = \alpha_{F C} \circ H \eta_C$. However, we observe that
\begin{align*}
\beta_C
& = \beta_C \circ \epsilon'_{F' H C} \circ F' \eta'_{H C} \\
& = \epsilon'_{K F C} \circ F' G' \beta_C \circ F' \eta'_{H C} \\
& = \epsilon'_{K F C} \circ F' \alpha_{F C} \circ F' H \eta_C
\end{align*}
so $\beta_C$ is precisely the component at $C$ of the mate of $\alpha$. Thus $\beta_C$ is an isomorphism for all $C$ if and only if the Beck--Chevalley condition holds.
\end{proof}

\begin{dfn}
Let $\kappa$ be a regular cardinal in a universe $\mbfU$, and let $\mbfUplus$ be a universe with $\mbfU \subseteq \mbfUplus$. A \strong{$\tuple{\kappa, \mbfU, \mbfUplus}$-accessible extension} is a $\tuple{\kappa, \mbfU}$-accessible functor $i : \mathcal{C} \to \succof{\mathcal{C}}$ such that
\begin{itemize}
\item $\mathcal{C}$ is a $\kappa$-accessible $\mbfU$-category,

\item $\succof{\mathcal{C}}$ is a $\kappa$-accessible $\mbfUplus$-category,

\item $i$ sends $\tuple{\kappa, \mbfU}$-compact objects in $\mathcal{C}$ to $\tuple{\kappa, \mbfUplus}$-compact objects in $\succof{\mathcal{C}}$, and

\item the functor $\Kompakt[\kappa][\mbfU]{\mathcal{C}} \to \Kompakt[\kappa][\mbfUplus]{\succof{\mathcal{C}}}$ so induced by $i$ is fully faithful and essentially surjective on objects.
\end{itemize}
\end{dfn}

\begin{remark}
Let $\mathbb{B}$ be a $\mbfU$-small category in which idempotents split. Then the $\tuple{\kappa, \mbfU}$-accessible functor $\Ind[\kappa][\mbfU]{\mathbb{B}} \to \Ind[\kappa][\mbfUplus]{\mathbb{B}}$ obtained by extending the embedding $\gamma^+ : \mathbb{B} \to \Ind[\kappa][\mbfUplus]{\mathbb{B}}$ along $\gamma : \mathbb{B} \to \Ind[\kappa][\mbfU]{\mathbb{B}}$ is a $\tuple{\kappa, \mbfU, \mbfUplus}$-accessible extension, by \autoref{prop:compact.objects.in.accessible.categories}. The classification theorem (\ref{thm:classification.of.accessible.categories}) implies all examples of $\tuple{\kappa, \mbfU, \mbfUplus}$-accessible extensions are essentially of this form.
\end{remark}

\begin{prop}
\label{prop:accessible.extension.properties}
\needspace{2.5\baselineskip}
Let $i : \mathcal{C} \to \succof{\mathcal{C}}$ be a $\tuple{\kappa, \mbfU, \mbfUplus}$-accessible extension.
\begin{enumerate}[(i)]
\item $\mathcal{C}$ is a locally $\kappa$-presentable $\mbfU$-category if and only if $\succof{\mathcal{C}}$ is a locally $\kappa$-pre\-sentable $\mbfUplus$-category.

\item The functor $i : \mathcal{C} \to \succof{\mathcal{C}}$ is fully faithful.

\item If $B : \mathcal{J} \to \mathcal{C}$ is any diagram (not necessarily $\mbfU$-small) and $\mathcal{C}$ has a limit for $B$, then $i$ preserves this limit.
\end{enumerate}
\end{prop}
\begin{proof}
(i). If $\mathcal{C}$ is a locally $\kappa$-presentable $\mbfU$-category, then $\Kompakt[\kappa][\mbfU]{\mathcal{C}}$ has colimits for all $\kappa$-small diagrams, so $\Kompakt[\kappa][\mbfUplus]{\succof{\mathcal{C}}}$ also has colimits for all $\kappa$-small diagrams. The classification theorem (\ref{thm:classification.of.accessible.categories}) then implies $\succof{\mathcal{C}}$ is a locally $\kappa$-presentable $\mbfUplus$-category. Reversing this argument proves the converse.

\bigskip\noindent
(ii). Let $A : \mathbb{I} \to \mathcal{C}$ and $B : \mathbb{J} \to \mathcal{C}$ be two $\mbfU$-small $\kappa$-filtered diagrams of $\tuple{\kappa, \mbfU}$-\allowhyphens compact objects in $\mathcal{C}$. Then,
\begin{multline*}
\Hom[\mathcal{C}]{\indlim_{\mathbb{I}} A}{\indlim_{\mathbb{J}} B}
  \cong \prolim_{\mathbb{I}} \indlim_{\mathbb{J}} \Hom[\mathcal{C}]{A}{B}
  \cong \prolim_{\mathbb{I}} \indlim_{\mathbb{J}} \Hom[\succof{\mathcal{C}}]{i A}{i B} \\
  \cong \Hom[\succof{\mathcal{C}}]{\indlim_{\mathbb{I}} i A}{\indlim_{\mathbb{J}} i B}
  \cong \Hom[\succof{\mathcal{C}}]{i \indlim_{\mathbb{I}} A}{i \indlim_{\mathbb{J}} B}
\end{multline*}
because $i$ is $\tuple{\kappa, \mbfU}$-accessible and is fully faithful on the subcategory $\Kompakt[\kappa][\mbfU]{\mathcal{C}}$, and therefore $i : \mathcal{C} \to \succof{\mathcal{C}}$ itself is fully faithful. Note that this hinges crucially on \autoref{thm:foundations:stability.of.limits.and.colimits.of.sets}.

\bigskip\noindent
(iii). Let $B : \mathcal{J} \to \mathcal{C}$ be any diagram. We observe that, for any $\tuple{\kappa, \mbfU}$-compact object $C$ in $\mathcal{C}$,
\begin{align*}
\Hom[\succof{\mathcal{C}}]{i C}{i \prolim_{\mathcal{J}} B} 
& \cong \Hom[\mathcal{C}]{C}{\prolim_{\mathcal{J}} B}
&& \text{because $i$ is fully faithful} \\
& \cong \prolim_{\mathcal{J}} \Hom[\mathcal{C}]{C}{B}
&& \text{by definition of limit} \\
& \cong \prolim_{\mathcal{J}} \Hom[\succof{\mathcal{C}}]{i C}{i B}
&& \text{because $i$ is fully faithful}
\end{align*}
but we know the restricted Yoneda embedding $\succof{\mathcal{C}} \to \Func{\op{\Kompakt[\kappa][\mbfU]{\mathcal{C}}}}{\cat{\Setplus}}$ is fully faithful, so this is enough to conclude that $i \prolim_{\mathcal{J}} B$ is the limit of $i B$ in $\succof{\mathcal{C}}$.
\end{proof}

\begin{remark}
Similar methods show that any fully faithful functor $\mathcal{C} \to \succof{\mathcal{C}}$ satisfying the four bulleted conditions in the definition above is necessarily $\tuple{\kappa, \mbfU}$-accessible.
\end{remark}

\begin{lem}
\label{lem:special.left.Beck-Chevalley.condition.for.accessible.extensions}
\needspace{3\baselineskip}
Let $\mbfU$ and $\mbfUplus$ be universes, with $\mbfU \in \mbfUplus$, and let $\kappa$ be a regular cardinal in $\mbfU$. Suppose:
\begin{itemize}
\item $\mathcal{C}$ and $\mathcal{D}$ are locally $\kappa$-presentable $\mbfU$-categories.

\item $\succof{\mathcal{C}}$ and $\succof{\mathcal{D}}$ are locally $\kappa$-presentable $\mbfUplus$-categories.

\item $i : \mathcal{C} \to \succof{\mathcal{C}}$ and $j : \mathcal{D} \to \succof{\mathcal{D}}$ are $\tuple{\kappa, \mbfU, \mbfUplus}$-accessible extensions.
\end{itemize}
Given a strictly commutative diagram of the form below,
\[
\begin{tikzcd}
\mathcal{D} \dar[swap]{G} \rar{j} &
\succof{\mathcal{D}} \dar{G^+} \\
\mathcal{C} \rar[swap]{i} &
\succof{\mathcal{C}}
\end{tikzcd}
\]
where $G$ is $\tuple{\kappa, \mbfU}$-accessible, $G^+$ is $\tuple{\kappa, \mbfUplus}$-accessible, if both have left adjoints, then the diagram satisfies the left Beck--Chevalley condition.
\end{lem}
\begin{proof}
Let $C$ be a $\tuple{\kappa, \mbfU}$-compact object in $\mathcal{C}$. \Autoref{prop:accessible.extension.properties} says that $i : \mathcal{C} \to \succof{\mathcal{C}}$ and $j : \mathcal{D} \to \succof{\mathcal{D}}$ preserve limits, so  by inspecting the proof of \autoref{thm:accessible.adjoints}, we see that the functor $\commacat{C}{G} \to \commacat{i C}{\succof{G}}$ induced by $j$ preserves initial objects. As in the proof of \autoref{lem:local.Beck-Chevalley.condition}, this implies the component at $C$ of the left Beck--Chevalley natural transformation $F^+ i \hoto j F$ is an isomorphism; but $\mathcal{C}$ is generated by $\Kompakt[\kappa][\mbfU]{\mathcal{C}}$ and the functors $F, F^+, i, j$ all preserve colimits for $\mbfU$-small $\kappa$-filtered diagrams, so in fact $F^+ i \hoto j F$ is a natural isomorphism.
\end{proof}

\begin{prop}
\label{prop:accessible.extensions.preserve.colimits}
If $i : \mathcal{C} \to \succof{\mathcal{C}}$ is a $\tuple{\kappa, \mbfU, \mbfUplus}$-accessible extension and $\mathcal{C}$ is a locally $\kappa$-presentable $\mbfU$-category, then $i$ preserves colimits for all $\mbfU$-small diagrams in $\mathcal{C}$.
\end{prop}
\begin{proof}
It is well-known that a functor preserves colimits for all $\mbfU$-small diagrams if and only if it preserves coequalisers for all parallel pairs and coproducts for all $\mbfU$-small families; but coproducts for $\mbfU$-small families can be constructed in a uniform way using coproducts for $\kappa$-small families and colimits for $\mbfU$-small $\kappa$-filtered diagrams, so it is enough to show that $i : \mathcal{C} \to \succof{\mathcal{C}}$ preserves all colimits for $\kappa$-small diagrams, since $i$ is already $\tuple{\kappa, \mbfU}$-accessible.

Let $\mathbb{D}$ be a $\kappa$-small category. Recalling \autoref{prop:limits.and.colimits}, our problem amounts to showing that the diagram
\[
\begin{tikzcd}
\mathcal{C} \dar[swap]{\Delta} \rar{i} &
\succof{\mathcal{C}} \dar{\Delta^+} \\
\Func{\mathbb{D}}{\mathcal{C}} \rar[swap]{i_*} &
\Func{\mathbb{D}}{\succof{\mathcal{C}}}
\end{tikzcd}
\]
satisfies the left Beck--Chevalley condition. It is clear that $i_*$ is fully faithful. Colimits for $\mbfU$-small diagrams in $\Func{\mathbb{D}}{\mathcal{C}}$ and in $\Func{\mathbb{D}}{\succof{\mathcal{C}}}$ are computed componentwise, so $\Delta$ and $i_*$ are certainly $\tuple{\kappa, \mbfU}$-accessible, and $\Delta^+$ is $\tuple{\kappa, \mbfUplus}$-accessible. Using \autoref{prop:compact.objects.in.diagram.categories}, we see that $i_*$ is also a $\tuple{\kappa, \mbfU, \mbfUplus}$-accessible extension, so we apply the lemma above to conclude that the left Beck--Chevalley condition is satisfied.
\end{proof}

\begin{thm}[Stability of accessible adjoint functors]
Let $\mbfU$ and $\mbfUplus$ be universes, with $\mbfU \in \mbfUplus$, and let $\kappa$ and $\lambda$ be regular cardinals in $\mbfU$, with $\kappa \le \lambda$. Suppose:
\begin{itemize}
\item $\mathcal{C}$ is a locally $\kappa$-presentable $\mbfU$-category.

\item $\mathcal{D}$ is a locally $\lambda$-presentable $\mbfU$-category.

\item $\succof{\mathcal{C}}$ is a locally $\kappa$-presentable $\mbfUplus$-category.

\item $\succof{\mathcal{D}}$ is a locally $\lambda$-presentable $\mbfUplus$-category.
\end{itemize}
Let $i : \mathcal{C} \to \succof{\mathcal{C}}$ be a $\tuple{\kappa, \mbfU, \mbfUplus}$-accessible extension and let $j : \mathcal{D} \to \succof{\mathcal{D}}$ be a fully faithful functor.
\begin{enumerate}[(i)]
\item Given a strictly commutative diagram of the form below,
\[
\begin{tikzcd}
\mathcal{D} \dar[swap]{G} \rar{j} &
\succof{\mathcal{D}} \dar{G^+} \\
\mathcal{C} \rar[swap]{i} &
\succof{\mathcal{C}}
\end{tikzcd}
\]
where $G$ is $\tuple{\lambda, \mbfU}$-accessible and $G^+$ is $\tuple{\lambda, \mbfUplus}$-accessible, if both have left adjoints and $j$ is a $\tuple{\lambda, \mbfU, \mbfUplus}$-\allowhyphens accessible extension, then the diagram satisfies the left Beck--Chevalley condition.

\item Given a strictly commutative diagram of the form below,
\[
\begin{tikzcd}
\mathcal{C} \dar[swap]{F} \rar{i} &
\succof{\mathcal{C}} \dar{F^+} \\
\mathcal{D} \rar[swap]{j} &
\succof{\mathcal{D}}
\end{tikzcd}
\]
if both $F$ and $F^+$ have right adjoints, then the diagram satisfies the right Beck--Chevalley condition.
\end{enumerate}
\end{thm}
\begin{proof}
(i). The proof is essentially the same as \autoref{lem:special.left.Beck-Chevalley.condition.for.accessible.extensions}, though we have to use \autoref{prop:accessible.extensions.preserve.colimits} to ensure that $j$ preserves colimits for all $\mbfU$-small $\kappa$-filtered diagrams in $\mathcal{C}$.

\bigskip\noindent
(ii). Let $D$ be any object in $\mathcal{D}$. Inspecting the proof of \autoref{thm:accessible.adjoints}, we see that our hypotheses, plus the fact that $i$ preserves colimits for all $\mbfU$-small diagrams in $\mathcal{C}$, imply that the functor $\commacat{F}{D} \to \commacat{\succof{F}}{j D}$ induced by $i$ preserves terminal objects. Thus \autoref{lem:local.Beck-Chevalley.condition} implies that the diagram satisfies the right Beck--Chevalley condition.
\end{proof}

\begin{thm}
\label{thm:accessible.extension.properties}
Let $i : \mathcal{C} \to \succof{\mathcal{C}}$ be a $\tuple{\kappa, \mbfU, \mbfUplus}$-accessible extension and let $\mathcal{C}$ be a locally $\kappa$-presentable $\mbfU$-category.
\begin{enumerate}[(i)]
\item If $\lambda$ is a regular cardinal in $\mbfU$ and $\kappa \le \lambda$, then $i : \mathcal{C} \to \succof{\mathcal{C}}$ is also a $\tuple{\lambda, \mbfU, \mbfUplus}$-accessible extension.

\item If $\mu$ is the cardinality of $\mbfU$, then $i : \mathcal{C} \to \succof{\mathcal{C}}$ factors through the inclusion $\Kompakt[\mu][\mbfUplus]{\succof{\mathcal{C}}} \embedinto \succof{\mathcal{C}}$ as functor $\mathcal{C} \to \Kompakt[\mu][\mbfUplus]{\succof{\mathcal{C}}}$ that is (fully faithful and) essentially surjective on objects. 

\item The $\tuple{\mu, \mbfUplus}$-accessible functor $\Ind[\mu][\mbfUplus]{\mathcal{C}} \to \succof{\mathcal{C}}$ induced by $i : \mathcal{C} \to \succof{\mathcal{C}}$ is fully faithful and essentially surjective on objects.
\end{enumerate}
\end{thm}
\begin{proof}
(i). Since $i : \mathcal{C} \to \mathcal{C}^+$ is a $\tuple{\kappa, \mbfU}$-accessible functor, it is certainly also $\tuple{\lambda, \mbfU}$-accessible, by \autoref{lem:accessibility.index.of.locally.presentable.categories}. It is therefore enough to show that $i$ restricts to a functor $\Kompakt[\kappa][\mbfU]{\mathcal{C}} \to \Kompakt[\kappa][\mbfUplus]{\succof{\mathcal{C}}}$ that is (fully faithful and) essentially surjective on objects.

\Autoref{prop:compact.objects.in.locally.presentable.categories} says $\Kompakt[\lambda][\mbfU]{\mathcal{C}}$ is the smallest replete full subcategory of $\mathcal{C}$ that contains $\Kompakt[\kappa][\mbfU]{\mathcal{C}}$ and is closed in $\mathcal{C}$ under colimits for $\lambda$-small diagrams, therefore the replete closure of the image of $\Kompakt[\lambda][\mbfU]{\mathcal{C}}$ must be the smallest replete full subcategory of $\succof{\mathcal{C}}$ that contains $\Kompakt[\kappa][\mbfUplus]{\succof{\mathcal{C}}}$ and is closed in $\succof{\mathcal{C}}$ under colimits for $\lambda$-small diagrams, since $i$ is fully faithful and preserves colimits for all $\mbfU$-small diagrams. This proves the claim.

\bigskip\noindent
(ii). Since every object in $\mathcal{C}$ is $\tuple{\lambda, \mbfU}$-compact for some regular cardinal $\lambda < \mu$, claim (i) implies that the image of $i : \mathcal{C} \to \succof{\mathcal{C}}$ is contained in $\Kompakt[\mu][\mbfUplus]{\mathcal{C}}$. To show $i$ is essentially surjective onto $\Kompakt[\mu][\mbfUplus]{\mathcal{C}}$, we simply have to observe that the inaccessibility of $\mu$ (\autoref{prop:universe.nesting}) and \autoref{prop:compact.objects.in.locally.presentable.categories} imply that, for $C'$ any $\tuple{\mu, \mbfUplus}$-compact object in $\succof{\mathcal{C}}$, there exists a regular cardinal $\lambda < \mu$ such that $C'$ is also a $\tuple{\lambda, \mbfUplus}$-compact object, which reduces the question to claim (i).

\bigskip\noindent
(iii). This is an immediate corollary of claim (ii) and the classification theorem (\ref{thm:classification.of.accessible.categories}) applied to $\succof{\mathcal{C}}$, considered as a $\tuple{\mu, \mbfUplus}$-accessible category.
\end{proof}

\begin{remark}
Although the fact $i : \mathcal{C} \to \succof{\mathcal{C}}$ that preserves limits and colimits for all $\mbfU$-small diagrams in $\mathcal{C}$ is a formal consequence of the theorem above (via \eg \autoref{thm:ind-completion}), it is not clear whether the theorem can be proved without already knowing this.
\end{remark}

\begin{cor}
If $\mathbb{B}$ is a $\mbfU$-small category and has colimits for all $\kappa$-small diagrams, and $\mu$ is the cardinality of $\mbfU$, then the canonical $\tuple{\mu, \mbfUplus}$-accessible functor $\Ind[\mu][\mbfUplus]{\Ind[\kappa][\mbfU]{\mathbb{B}}} \to \Ind[\kappa][\mbfUplus]{\mathbb{B}}$ is fully faithful and essentially surjective on objects. \qed 
\end{cor}

\begin{thm}[Stability of pointwise Kan extensions]
\label{thm:stability.of.pointwise.Kan.extensions}
Let $\mbfUplus$ be a pre-universe, let $\mathcal{A}$ be a $\mbfUplus$-small category, let $\mathcal{C}, \mathcal{D}, \succof{\mathcal{C}}, \succof{\mathcal{D}}$ be locally $\mbfUplus$-small categories, let $F : \mathcal{A} \to \mathcal{C}$ and $G : \mathcal{A} \to \mathcal{D}$ be functors, and let $i : \mathcal{C} \to \succof{\mathcal{C}}$ and $j : \mathcal{D} \to \succof{\mathcal{D}}$ be fully faithful functors. Consider the following (not necessarily commutative) diagram:
\[
\begin{tikzcd}
\mathcal{A} \dar[swap]{F} \rar{G} &
\mathcal{D} \rar{j} &
\succof{\mathcal{D}} \\
\mathcal{C} \dar[swap]{i} \urar[swap]{H} \\
\succof{\mathcal{C}} \arrow[swap]{uurr}{\succof{H}}
\end{tikzcd}
\]
\begin{enumerate}[(i)]
\item If $\succof{H}$ is a pointwise right Kan extension of $j G$ along $i F$, and $\succof{H} i \cong j H$, then $H$ is a pointwise right Kan extension of $G$ along $F$.

\item Suppose $\mathcal{C}$ is $\mbfUplus$-small and $j H$ is a pointwise right Kan extension of $j G$ along $F$. If $\succof{H}$ is a pointwise right Kan extension of $j H$ along $i$, then the counit $\succof{H} i \hoto j H$ is a natural isomorphism, and $\succof{H}$ is also a pointwise right Kan extension of $j G$ along $i F$; conversely, if $\succof{H}$ is a pointwise right Kan extension of $j G$ along $i F$, then it is also a pointwise right Kan extension of $j H$ along $i$.

\item If $\mbfU$ is a pre-universe such that $\mathcal{A}$ is $\mbfU$-small and $j$ preserves limits for all $\mbfU$-small diagrams, and $H$ is a pointwise right Kan extension of $G$ along $F$, then a pointwise right Kan extension of $j G$ along $i F$ can be computed as a pointwise right Kan extension of $j H$ along $i$ (if either one exists).
\end{enumerate}
Dually:
\begin{enumerate}[(i\prime)]
\item If $\succof{H}$ is a pointwise left Kan extension of $j G$ along $i F$, and $\succof{H} i \cong j H$, then $H$ is a pointwise left Kan extension of $G$ along $F$.

\item Suppose $\mathcal{C}$ is $\mbfUplus$-small and $j H$ is a pointwise left Kan extension of $j G$ along $F$. If $j H$ is a pointwise left Kan extension of $j G$ along $F$, and $\succof{H}$ is a pointwise left Kan extension of $j H$ along $i$, then the unit $j H \hoto \succof{H} i$ is a natural isomorphism, and $\succof{H}$ is also a pointwise left Kan extension of $j G$ along $i F$.

\item If $\mbfU$ is a pre-universe such that $\mathcal{A}$ is $\mbfU$-small and $j$ preserves colimits for all $\mbfU$-small diagrams, and $H$ is a pointwise left Kan extension of $G$ along $F$, then a pointwise left Kan extension of $j G$ along $i F$ can be computed as a pointwise left Kan extension of $j H$ along $i$ (if either one exists).
\end{enumerate}
\end{thm}
\begin{proof}
(i). We have the following explicit description of $\succof{H} : \succof{\mathcal{C}} \to \succof{\mathcal{D}}$ as a weighted limit:\footnote{See Theorem 4.6 in \citep{Kelly:2005}.}
\[
\succof{H} \argp{C'} \cong \kwlim[\mathcal{A}]{\Hom[\succof{\mathcal{C}}]{C'}{i F}}{j G} 
\]
Since $i$ is fully faithful, the weights $\Hom[\mathcal{C}]{C}{F}$ and $\Hom[\succof{\mathcal{C}}]{i C}{i F}$ are naturally isomorphic, hence,
\[
j H \argp{C} \cong \succof{H} \argp{i C} \cong \kwlim[\mathcal{A}]{\Hom[\succof{\mathcal{C}}]{i C}{i F}}{j G} \cong \kwlim[\mathcal{A}]{\Hom[\mathcal{C}]{C}{F}}{j G}
\]
but, since $j$ is fully faithful, $j$ reflects \emph{all} weighted limits, therefore $H$ must be a pointwise right Kan extension of $G$ along $F$.

\bigskip\noindent
(ii). Let $\cat{\Setplus}$ be the category of $\mbfUplus$-sets. Using the interchange law\footnote{See \citep[\Chap IX, \Sect 8]{CWM}.} and the end version of the Yoneda lemma, we obtain the following natural bijections:
\begin{align*}
\Hom[\succof{\mathcal{D}}]{D'}{\succof{H} \argp{C'}}
& \cong \Hom[\succof{\mathcal{D}}]{D'}{\kwlim[\mathcal{C}]{\Hom[\succof{\mathcal{C}}]{C'}{i}}{j H}} \\
& \cong \int_{C : \mathcal{C}} \Hom[\Setplus]{\Hom[\succof{\mathcal{C}}]{C'}{i C}}{\Hom[\succof{\mathcal{D}}]{D'}{j H C}} \\
& \cong \int_{C : \mathcal{C}} \Hom[\Setplus]{\Hom[\succof{\mathcal{C}}]{C'}{i C}}{\Hom[\succof{\mathcal{D}}]{D'}{\kwlim[\mathcal{A}]{\Hom[\mathcal{C}]{C}{F}}{j G}}} \\
& \cong \int_{C : \mathcal{C}} \int_{A : \mathcal{A}} \Hom[\Setplus]{\Hom[\succof{\mathcal{C}}]{C'}{i C}}{\Hom[\Setplus]{\Hom[\mathcal{C}]{C}{F A}}{\Hom[\succof{\mathcal{D}}]{D'}{j G A}}} \\
& \cong \int_{C : \mathcal{C}} \int_{A : \mathcal{A}} \Hom[\Setplus]{\Hom[\mathcal{C}]{C}{F A}}{\Hom[\Setplus]{\Hom[\succof{\mathcal{C}}]{C'}{i C}}{\Hom[\succof{\mathcal{D}}]{D'}{j G A}}} \\
& \cong \int_{A : \mathcal{A}} \int_{C : \mathcal{C}} \Hom[\Setplus]{\Hom[\mathcal{C}]{C}{F A}}{\Hom[\Setplus]{\Hom[\succof{\mathcal{C}}]{C'}{i C}}{\Hom[\succof{\mathcal{D}}]{D'}{j G A}}} \\
& \cong \int_{A : \mathcal{A}} \Hom[\Setplus]{\Hom[\succof{\mathcal{C}}]{C'}{i F A}}{\Hom[\succof{\mathcal{D}}]{D'}{j G A}} \\
& \cong \Hom[\succof{\mathcal{D}}]{D'}{\kwlim[\mathcal{A}]{\Hom[\succof{\mathcal{C}}]{C'}{i F}}{j G}}
\end{align*}
Thus, $\succof{H}$ is a pointwise right Kan extension of $j G$ along $i F$ if and only if $\succof{H}$ is a pointwise right Kan extension of $j H$ along $i$. The fact that the counit $\succof{H} i \hoto j H$ is a natural isomorphism follows from the fact that $i$ is fully faithful.\footnote{See Proposition 4.23 in \citep{Kelly:2005}.}

\bigskip\noindent
(iii). Apply the fact that pointwise Kan extensions are preserved by functors that preserve sufficiently large limits to claim (ii).
\end{proof}

\begin{cor}
\needspace{3\baselineskip}
Let $\mbfU$ and $\mbfUplus$ be universes, with $\mbfU \in \mbfUplus$, and let $\kappa$ and $\lambda$ be regular cardinals in $\mbfU$. Suppose:
\begin{itemize}
\item $\mathcal{C}$ is a locally $\kappa$-presentable $\mbfU$-category.

\item $\mathcal{D}$ is a locally $\lambda$-presentable $\mbfU$-category.

\item $\succof{\mathcal{C}}$ is a locally $\kappa$-presentable $\mbfUplus$-category.

\item $\succof{\mathcal{D}}$ is a locally $\lambda$-presentable $\mbfUplus$-category.
\end{itemize}
Let $F : \mathcal{A} \to \mathcal{C}$ and $G : \mathcal{A} \to \mathcal{D}$ be functors, let $i : \mathcal{C} \to \succof{\mathcal{C}}$ be a $\tuple{\kappa, \mbfU, \mbfUplus}$-\allowhyphens accessible extension, and let $j : \mathcal{D} \to \succof{\mathcal{D}}$ be a $\tuple{\lambda, \mbfU, \mbfUplus}$-accessible extension. Consider the following (not necessarily commutative) diagram:
\[
\begin{tikzcd}
\mathcal{A} \dar[swap]{F} \rar{G} &
\mathcal{D} \rar{j} &
\succof{\mathcal{D}} \\
\mathcal{C} \dar[swap]{i} \urar[swap]{H} \\
\succof{\mathcal{C}} \arrow[swap]{uurr}{\succof{H}}
\end{tikzcd}
\]
\begin{enumerate}[(i)]
\item Assuming $\mathcal{A}$ is $\mbfU'$-small for some pre-universe $\mbfU'$, if $H$ is a  pointwise right Kan extension of $G$ along $F$, then $j H$ is a pointwise right Kan extension of $j G$ along $F$; and if $\succof{H}$ is a  pointwise right Kan extension of $j H$ along $i$, then $\succof{H}$ is also a pointwise right Kan extension of $j G$ along $i F$.

\item Assuming $\mathcal{A}$ is $\mbfU$-small, if $H$ is a pointwise left Kan extension of $G$ along $F$, then $j H$ is a pointwise left Kan extension of $j G$ along $F$; and if $\succof{H}$ is a pointwise left Kan extension of $j H$ along $i$, then $\succof{H}$ is also a pointwise left Kan extension of $j G$ along $i F$.
\end{enumerate}
\end{cor}
\begin{proof}
Use the theorem and the fact that $i$ and $j$ preserve limits for \emph{all} diagrams (\autoref{prop:accessible.extension.properties}) and colimits for $\mbfU$-small diagrams (\autoref{prop:accessible.extensions.preserve.colimits}).
\end{proof}

%% file: Content-99.tex
\section{Future work}

\makenumpar
One of motivating questions behind this paper was the following:
\begin{quote}
Let $\mathbb{B}$ be a $\mbfU$-small category with colimits for $\kappa$-small diagrams. Given a combinatorial model structure on $\mathcal{M} = \Ind[\kappa][\mbfU]{\mathbb{B}}$, does it extend to a combinatorial model structure on $\succof{\mathcal{M}} = \Ind[\kappa][\mbfUplus]{\mathbb{B}}$?
\end{quote}
More precisely, if $\mathcal{I}$ is a $\mbfU$-set of generating cofibrations and $\mathcal{J}$ is a $\mbfU$-set of generating trivial cofibrations in $\mathcal{M}$, is there a model structure on $\succof{\mathcal{M}}$ that is cofibrantly generated by $\mathcal{I}$ and $\mathcal{J}$?

\needspace{2.5\baselineskip}
Using \autoref{prop:accessible.extensions.preserve.colimits} and \autoref{thm:accessible.extension.properties}, it is easy to establish these facts:
\begin{enumerate}[(i)]
\item The inclusion $\mathcal{M} \embedinto \succof{\mathcal{M}}$ preserves the functorial factorisations constructed by Quillen's small object argument.\footnote{See \eg Lemma 3 in \citep[\Chap II, \Sect 3]{Quillen:1967}.}

\item Every $\mathcal{I}$-injective morphism in $\succof{\mathcal{M}}$ is also a $\mathcal{J}$-injective morphism.

\item Every $\mathcal{J}$-cofibration in $\succof{\mathcal{M}}$ is also an $\mathcal{I}$-cofibration.
\end{enumerate}
Now, define $\succof{\mathcal{W}}$ to be the collection of all morphisms in $\succof{\mathcal{M}}$ of the form $q \circ j$ for a $\mathcal{J}$-cofibration $j$ and a $\mathcal{I}$-injective morphism $q$. Then:
\begin{enumerate}[resume*]
\item Any $\mathcal{I}$-cofibration in $\succof{\mathcal{M}}$ that is also in $\succof{\mathcal{W}}$ is a (retract of a) $\mathcal{J}$-cofibration.

\item Any $\mathcal{J}$-injective morphism in $\succof{\mathcal{M}}$ that is also in $\succof{\mathcal{W}}$ is a (retract of a) $\mathcal{I}$-injective morphism.
\end{enumerate}
Thus, as soon as we know that $\succof{\mathcal{W}}$ has the 2-out-of-3 property in $\succof{\mathcal{M}}$, we would have a (cofibrantly-generated) model structure on $\succof{\mathcal{M}}$ extending the model structure on $\mathcal{M}$;\footnote{The 2-out-of-3 property plus factorisations imply that $\succof{\mathcal{W}}$ is closed under retracts; see Lemma 14.2.5 in \citep{May-Ponto:2012}.} unfortunately, it is not clear to the author how this can be done.\footnote{The question has since been answered in the affirmative: see \citep{Low:2014a}.}

\makenumpar
It would also be interesting to investigate strengthenings of the results in this paper for weaker notions of universe. For example, let us say that $\mbfU$ is a \strong{weak universe} if $\mbfU$ satisfies axioms 1--3 and 5 of \autoref{dfn:universe} \emph{plus} the following axiom:
\begin{enumerate}
\item[4\textsuperscript{$-$}.] If $x \in \mbfU$, then $\bigcup_{y \in x} y \in \mbfU$.
\end{enumerate}
In Mac Lane set theory, if $\mbfU$ is a weak universe, then $\mbfU$ is a model of Zermelo set theory\footnote{--- \ie Zermelo--Fraenkel set theory \emph{minus} replacement.} with global choice and so the category of $\mbfU$-sets is a model of Lawvere's ETCS. Moreover, in plain ZFC, every set is a member of some weak universe: indeed, for every limit ordinal $\alpha > \omega$, the set $\mbfV_\alpha$ is a weak universe.\footnote{However, note that $\mbfV_{\omega + \omega}$ may not exist in Mac Lane set theory!} If it turns out that accessible extensions are still well-behaved in this context, then we would have an adequate framework for studying the theory of categories by category-theoretic means without having to appeal to large cardinal axioms.

\makenumpar
On the other hand, there are also less promising notions of universe. From the perspective of a set theorist, one natural question to ask is how category theory in an inner model of ZFC, such as Gödel's constructible universe $\mbfL$, relates to category theory in the true universe $\mbfV$. However, if $\cat{\Set}_{\mbfL}$ is the (meta)category of $\mbfL$-sets and $\cat{\Set}_{\mbfV}$ is the (meta)category of $\mbfV$-sets, then the inclusion $\cat{\Set}_{\mbfL} \embedinto \cat{\Set}_{\mbfV}$ need not be full, or even conservative. This appears to be a severe obstacle to the deployment of category-theoretic tools in solving this problem.